\documentclass{article}

\usepackage{amscd}
\usepackage{amssymb} 
\usepackage{amsmath} 
\usepackage{latexsym} 
\usepackage{theorem} 
\usepackage{color}
\usepackage{epsfig}
\usepackage{psfrag}
\usepackage{makeidx}
\usepackage{eepic,slashbox}

\theorembodyfont{\itshape} 

\newtheorem{theorem}{Theorem}
\newtheorem{lemma}[theorem]{Lemma}

\newtheorem{corollary}[theorem]{Corollary}
\newtheorem{proposition}[theorem]{Proposition}

\newtheorem{definition}[theorem]{Definition}

\newtheorem{remark}[theorem]{Remark}
\newenvironment{proof}{{\par\addvspace{0.1cm}\noindent \bf Proof. }}{\hfill$\Box$\par\medskip} 
 
\def\RR{\mathbb{R}}
\def\CC{\mathbb{C}}

\def\vect#1{\mbox{\boldmath $#1$}}

\title{Pompeiu's theorem and the moduli space of triangles}

\author{Jun O'Hara\footnote{Supported by JSPS KAKENHI Grant Number 16K05136.}}

\begin{document}

\maketitle

\begin{abstract} 
We introduce a kind of converse of Pompeiu's theorem. Fix an equilateral triangle $\triangle A_0B_0C_0$, then for any triangle $\triangle ABC$ there is a unique point $P$ inside the circumcircle $\Gamma_0$ of $\triangle A_0B_0C_0$ such that a triangle with edge lengths $PA_0, PB_0$, and $PC_0$ is similar to $\triangle ABC$. It follows that an open disc inside $\Gamma_0$ can be considered as a moduli space of similarity classes of triangles.  We show that it is essentially equivalent to another moduli space based on a shape function of triangles which has been used in preceding studies. 
\end{abstract}



\section{Introduction}

We identify $\RR^2$ with $\CC$ and fix an equilateral triangle $\triangle A_0B_0C_0$, where $A_0=i,B_0=-e^{\pi i/6}$ and $C_0=e^{-\pi i/6}$. 
Let $\Gamma_0$ be the circumcircle and $\mathcal{D}$ a unit open disc with center the origin. 
F. van Schooten's theorem states that for any point $P$ on $\Gamma_0$, edges with lengths $|P-A_0|, |P-B_0|$, and $|P-C_0|$ form a degenerate triangle, i.e. the greatest is the sum of the other two (\cite{F} pp. 62\,--\,64). On the other hand, Pompeiu's theorem states 

\begin{theorem}\label{thm_Pompeiu} 
{\rm (Pompeiu \cite{P})} For any point $P$ in $\RR^2\setminus\Gamma_0$ there is a triangle with side lengths $|P-A_0|, |P-B_0|$, and $|P-C_0|$. 
\end{theorem}

We show a kind of converse of this theorem: for any triangle $\triangle ABC$ there is a unique point $P$ in $\mathcal{D}$ such that if a triangle $\triangle A'B'C'$ satisfies $|B'-C'|=|P-A_0|, |C'-A'|=|P-B_0|$, and $|A'-B'|=|P-C_0|$ then $\triangle A'B'C'$ is similar to $\triangle ABC$ (Theorem \ref{main_thm1} (1)). 
Thus the open disc $\mathcal{D}$ can be considered as the moduli space of the space of similarity classes of triangles, denoted by $\widetilde{\mathcal{S}}$. 
In other words, we obtain a bijection $\phi_P$ from $\widetilde{\mathcal{S}}$ to $\mathcal{D}$. 
Let us call the pair $(\mathcal{D}, \phi_P)$ the {\em Pompeiu moduli space of similarity classes of triangles} and denote it by $\mathcal{D}_P$. 
We remark that the center of $\mathcal{D}$ corresponds to equilateral triangles and that as a point in $\mathcal{D}$ approaches the boundary circle $\Gamma_0$ the corresponding triangle tends to a degenerate one. 

On the other hand, there is another way to identify $\widetilde{\mathcal{S}}$ with $\mathcal{D}$, which has appeared in \cite{L1, NO0, HMS, H}. 
We identify $\RR^2$ with $\CC$ and assume that the vertices of triangles are labeled in an anti-clockwise order. 
Given a triangle $\triangle ABC$ let $\sigma=\sigma(\triangle ABC)$ be the cross ratio 
\begin{equation}\label{def_sigma}
\sigma=(\infty,B;C,A)=\frac{(\infty-C)(B-A)}{(\infty-A)(B-C)}=\frac{A-B}{C-B}\in\CC. 
\end{equation}
This correspondence induces a bijection, denoted by the same symbol $\sigma$, from $\widetilde{\mathcal{S}}$ to the upper half plane $\mathcal{H}_+$. 
Put $\rho=e^{\pi i/3}$, the primitive $6$th root of $1$.  
Define a map from the set of triangles $\mathcal{T}$ to $\CC$ by 
\begin{equation}\label{def_phi}
\phi(\triangle ABC)=\frac{A+\rho^2B+\rho^{4}C}{A+\rho^{4}B+\rho^{2}C}
=\frac{\sigma-\rho}{\sigma-\bar\rho}\, ,
\end{equation}
which induces a bijection from $\widetilde{\mathcal{S}}$ to $\mathcal{D}$, denoted by the same symbol $\phi$. 
We remark that $\sigma$ is a modification of the {\em shape functions} given in \cite{L1}, $\phi^3$ has appeared in \cite{NO0,NO1} and $\phi$ in \cite{HMS,H}. 

We show that 
$\phi\circ{\phi_P}^{-1}$ can be expressed as $\phi\circ{\phi_P}^{-1}(z)=i\bar z$ 
(Theorem \ref{main_thm2}), namely, the moduli space of similarity classes of triangles obtained via the converse Pompeiu's theorem is essentially equivalent to that by the shape function. 

Let us introduce some corollaries. 
First, we can define two kinds of operations on the space of similarity classes of triangles, $R_\theta$ and $H_\lambda$, which serve as the rotation by angle $\theta$ and the homothety by factor $e^\lambda$ on the moduli space $\mathcal{D}_P$. 
Fixing the edge $BC$, the rotation $R_\theta$ can be given using a circle of Apollonius with foci $\Omega_A$ and $\overline{\Omega}_A$, where $\Omega_A$ and $\overline{\Omega}_A$ are two vertices of equilateral triangles which have $BC$ as one of the edges, whereas the homothety $H_\lambda$ can be given using a circle that passes through $\Omega_A, \overline{\Omega}_A$, and $A$ (Figure \ref{fig_operations}). 
We give M\"obius geometric meaning of $\theta$ and $\lambda$ as the arc-lengths of geodesics in the de Sitter space which can be identified with the set of oriented circles. 
We remark that the result of Nakamura and Oguiso \cite{NO0} implies that the rotation $R_\theta$ can also be realized by operations illustrated in Figures \ref{fig_T_q}, \ref{fig_S_q}, and \ref{fig_R_pq}, where the definitions will be stated later in the last section.

Next, suppose we work in $\RR^3$. 
Let $\Pi$ be a plane that contains $\triangle A_0B_0C_0$, $\Gamma_0$ the circumcircle of $\triangle A_0B_0C_0$, $O$ the center, and $\Sigma$ a sphere that has $\Gamma_0$ as equator. 
For a non-equilateral triangle $\triangle ABC$ put $S$ to be the set of points $P$ in $\RR^3$ such that a triangle $\triangle A'B'C'$ with $|B'-C'|=|P-A_0|, |C'-A'|=|P-B_0|$, and $|A'-B|'=|P-C_0|$ is similar to $\triangle ABC$. 
Then $S$ is a circle of Apollonius in a plane $\Pi^\perp$ through $O$ that is orthogonal to $\Pi$ with foci $\Pi^\perp\cap\Gamma_0$, in other words, $S$ is a circle which is invariant under both an inversion in $\Sigma$ and a reflection in $\Pi$.

\section{Preliminaries}
\subsection{Notations and definitions}\label{subsection_def}
We identify $\RR^2$ with $\CC$ and assume that the vertices of triangles are labeled anti-clockwisely. 

We say that two triangles $\triangle ABC$ and $\triangle A'B'C'$ are {\em similar}, denoted by $\triangle ABC\sim\triangle A'B'C'$, if $\triangle ABC$ can be mapped to $\triangle A'B'C'$ by a composition of homothety, rotation, and translation. 
In other words, our similarity means the one preserving the order of vertices and the orientation of the triangle.  

Let $\mathcal{T}$ be the set of triangles, $\widetilde{\mathcal{S}}$ the space of similarity classes of triangles, $\widetilde{\mathcal{S}}=\mathcal{T}/\sim$, and $\mbox{\rm pr}$ the projection from $\mathcal{T}$ to $\widetilde{\mathcal{S}}$. 
The similarity class of a triangle $\triangle ABC$ is denoted by $[\triangle ABC]$. 
Any triangle $\triangle ABC$ with $|B-C|=a, |C-A|=b$, and $|A-B|=c$ 
will be denoted by $\triangle (a,b,c)$. 

We use $\mathcal{D}$ for a unit open disc of $\CC$ with center the origin, $\Gamma_0=\partial\mathcal{D}$ for a unit circle with center the origin, and $\mathcal{H}_+$ (or $\mathcal{H}_-$) for the upper (or respectively lower) half plane of $\CC$. 

Let $\rho$ denote the primitive $6$th root of $1$, $\rho=e^{\pi i/3}$. 

A circle with center $\eta$ and radius $r$ will be denoted by $\Gamma(\eta,r)$. 
An inversion in a circle $\Gamma$ will be denoted by $I_{\Gamma}$. 
\[\displaystyle I_{\Gamma(\eta,r)}\colon z\mapsto \eta+\frac{r^2(z-\eta)}{{|z-\eta|}^2}=\eta+\frac{r^2}{\bar z-\bar\eta}\,. \]

\subsection{Apollonius circles and pencils of circles}\label{subsection_pencils}

Let $\Omega$ and $\bar\Omega$ be a pair of points. The set of points $P$ such that $|P-\Omega|:|P-\bar\Omega|$ is constant is a circle called a {\em circle of Apollonius with foci $\Omega$ and $\bar\Omega$}, which may be the bisector of $\Omega\bar\Omega$ as a special case. 
The set of circles of Apollonius is called the {\em Poncelet pencil} ({\em hyperbolic pencil}) of circles with {\em limit points} ({\em Poncelet points}) $\Omega$ and $\bar\Omega$. 
The set of circles that pass through $\Omega$ and $\bar\Omega$ is called a {\em pencil with base points $\Omega$ and $\bar\Omega$} or {\em elliptic pencil}. 
Every circle of Apollonius intersects every circle through $\Omega$ and $\bar\Omega$ orthogonally. 
By an inversion $I$ in a circle with center $\bar\Omega$, these two pencils are mapped to the set of concentric circles with center $I(\Omega)$ and the set of straight lines through $I(\Omega)$ (Figure \ref{pencil}). 

\begin{figure}[htbp]
\begin{center}
\includegraphics[width=.5\linewidth]{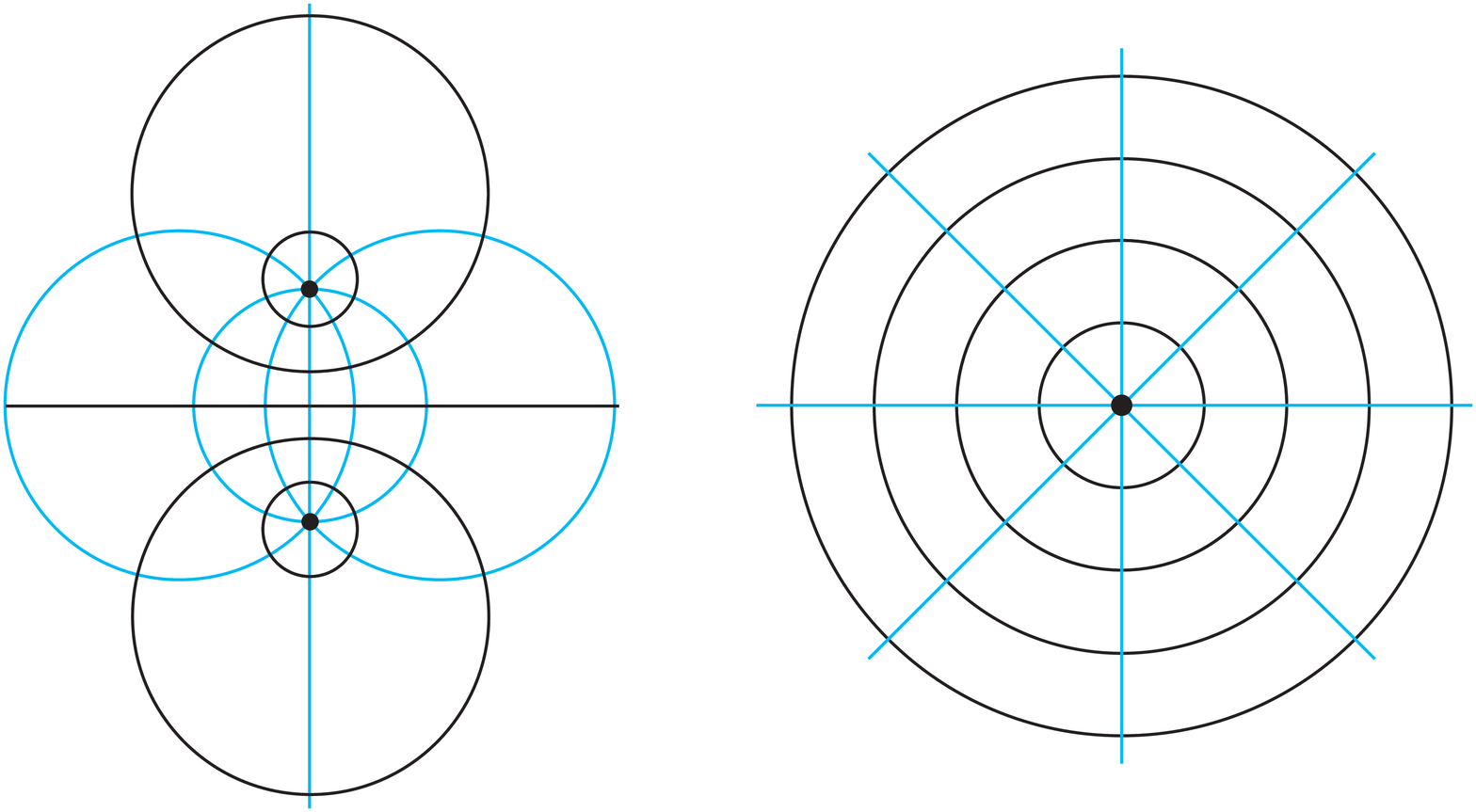}
\caption{A pair of a hyperbolic pencil (black) and an elliptic pencil (blue) of circles. }
\label{pencil}
\end{center}
\end{figure}

\section{Pompeiu moduli space}
%
We fix an equilateral triangle $\triangle A_0B_0C_0$, where $A_0=i,B_0=-e^{\pi i/6}$ and $C_0=e^{-\pi i/6}$, which has the unit circle $\Gamma_0=\partial\mathcal{D}$ as the circumcircle. 

\begin{theorem}\label{main_thm1}
Let $\triangle ABC$ be any triangle. 
\begin{enumerate}
\item There is a unique point $P$ in $\mathcal{D}$ such that 
$\triangle(|P-A_0|, |P-B_0|, |P-C_0|) \sim \triangle ABC$. 
\item When $\triangle ABC$ is not equilateral, there is a unique point $P'$ in the exterior of $\mathcal{D}$ that satisfies $\triangle(|P'-A_0|, |P'-B_0|, |P'-C_0|) \sim \triangle ABC$. 
\item When $\triangle ABC$ is not equilateral, the point $P'$ is the image of $P$ under an inversion in $\Gamma_0$. 
\end{enumerate}
\end{theorem}

\begin{proof}
Although (1) follows from the proof of Theorem \ref{main_thm2}, we give an independent proof here. 

(1), (2) 
Suppose $a=|B-C|, b=|C-A|$, and $c=|A-B|$. 
Let $\Gamma_b$ and $\Gamma_c$ be circles of Apollonius with foci $A_0,B_0$ and $A_0,C_0$ respectively given by \setlength\arraycolsep{1pt}
\[\begin{array}{rcl}
\Gamma_b&=&\{X\,:\,|X-A_0|:|X-B_0|=a:b\},\\[1mm]
\Gamma_c&=&\{X\,:\,|X-A_0|:|X-C_0|=a:c\}. 
\end{array}
\]
Then if there is a point $P$ that satisfies $|P-A_0|:|P-B_0|:|P-C_0|=a:b:c$ then $P$ is in the intersection of $\Gamma_b$ and $\Gamma_c$. 

Let $I=I_{\Gamma(A_0,\sqrt3)}$ be an inversion in a circle with center $A_0$ that passes through both $B_0$ and $C_0$. 
Then a simple computation shows that $I(\Gamma_b)$ is a circle with center $B_0$ and radius $\sqrt3\,b/a$, and $I(\Gamma_c)$ is a circle with center $C_0$ and radius $\sqrt3\,c/a$. 
Since three edges with lengths $a,b$, and $c$ form a triangle, $I(\Gamma_b)$ and $I(\Gamma_c)$ intersect in a pair of points, say $\widetilde P$ and $\widetilde P'$, with $\widetilde P$ (or $\widetilde P'$) being below (or respectively, above) the line $B_0C_0$. 
Taking the pre-image by $I$, it follows that $\Gamma_b$ and $\Gamma_c$ intersect in a pair points $I(\widetilde P)$ and $I(\widetilde P')$. 
Since $I(\mathcal{D})$ is the lower half plane below the line $B_0C_0$, $I(\widetilde P)$ is in $\mathcal{D}$ and $I(\widetilde P')$ in the exterior. 
Now we can put $P=I(\widetilde P)$ and $P'=I(\widetilde P')$. 
Remark that if $\triangle ABC$ is equilateral then $\widetilde P'=A_0$ and hence $I(\widetilde P')=\infty$. 
It completes the proof of (1) and (2). 

\medskip
(3) Let $I_0=I_{\Gamma_0}$ be an inversion in $\Gamma_0$. 
Recall that the distance between a pair of points satisfies 
$|I_0(P)-I_0(Q)|={|P|}^{-1}{|Q|}^{-1}|P-Q| $ for $P,Q\ne0$. Since $A_0, B_0$, and $C_0$ are invariant under $I_0$, 
\[
|I_0(P)-A_0|:|I_0(P)-B_0|:|I_0(P)-C_0|=|P-A_0|:|P-B_0|:|P-C_0|,
\]
which implies 
\[
\triangle(|I_0(P)-A_0|, |I_0(P)-B_0|, |I_0(P)-C_0|) \sim \triangle(|P-A_0|, |P-B_0|, |P- C_0|). 
\]
\end{proof}

\begin{definition} \rm 
Define a bijection $\phi_P$ from the space of similarity classes of triangles $\widetilde{\mathcal{S}}$ to the unit open disc $\mathcal{D}$ by 
\begin{equation}\label{def_phi_P}
{\phi_P}^{-1}(P)=\left[\triangle(|P-A_0|, |P-B_0|, |P-C_0|)\right], 
\end{equation}
and call the pair $(\mathcal{D}, \phi_P)$ the {\em Pompeiu moduli space} of similarity classes of triangles and denote it by $\mathcal{D}_P$. 
\end{definition}

\begin{remark}\rm 
The statement (3) can also be proved as follows. 

Let $I_{\overline{B_0C_0}}$ be the reflection in a line $\overline{B_0C_0}$. 
Since $\widetilde P'=I_{\overline{B_0C_0}}\,\big(\widetilde P\,\big)$ we have 
\[
P'=I(\widetilde P')=I_{\Gamma(A_0,\sqrt3)}\left(I_{\overline{B_0C_0}}\,\big(\widetilde P\,\big)\right)
=I_{I_{\Gamma(A_0,\sqrt3)}(\overline{B_0C_0})}\left(I_{\Gamma(A_0,\sqrt3)}\big(\widetilde P\,\big)\right)
=I_{\Gamma_0}(P). 
\]
\end{remark}

\begin{corollary}\label{cor_3-dim}
Let $\Pi$ be a plane in $\RR^3$, $\triangle A_0B_0C_0$ an equilateral triangle in $\Pi$, and $\Gamma_0$ the circumcircle of $\triangle A_0B_0C_0$. 
Then for any non-equilateral triangle $\triangle ABC$ the set $S$ of points $P$ in $\RR^3$ such that 
$\triangle(|P-A_0|, |P-B_0|, |P-C_0|)\sim\triangle ABC$ is a circle with diameter $PP'$ that intersects $\Pi$ orthogonally, where $P$ and $P'$ are the points given in Theorem \ref{main_thm1}. 

Let $\Gamma_{ABC}$ be this circle and $\Pi^\perp_{ABC}$ the plane that contains $\Gamma_{ABC}$. 
Then $\Gamma_{ABC}$ is a circle of Apollonius with foci $\Gamma_0\cap\Pi_{ABC}^\perp$. 
\end{corollary}

Thus, if we denote a sphere which has $\Gamma_0$ as the equator by $\Sigma$,  the set $\widetilde{\mathcal{S}}$ of similarity classes can be identified with the set of circles that are invariant under both the inversion $\Sigma$ and reflection in $\Pi$. 

\begin{proof}
Put $a=|B-C|, b=|C-A|$, and $c=|A-B|$. 
The set $S$ is the intersection of the set $S_1$ of points $P$ in $\RR^3$ such that $|P-A_0|:|P-B_0|=a:b$ and the set $S_2$ of points $P$ in $\RR^3$ such that $|P-A_0|:|P-C_0|=a:c$. 
Since $S_1\cap\Pi$ (or $S_2\cap\Pi$) is a circle of Apollonius with foci $A_0$ and $B_0$ (or $A_0$ and $C_0$ respectively), $S_1$ (or $S_2$) is a sphere with center on $\Pi$, which implies the first statement. 

The second statement is a consequence of Theorem \ref{main_thm1} (3). 
In fact, if we put $\Gamma_0\cap\Pi_{ABC}^\perp=\{W,W'\}$ then $|P-W|:|P-W'|=|P'-W|:|P'-W'|$. 
\end{proof}

Next we proceed to show that Pompeiu moduli space $\mathcal{D}_P$ is essentially equivalent to the moduli space $\mathcal{D}$ explained in Introduction. 

\begin{lemma} \label{key_lemma}
Let $\{a,b,c\}$ be a triple of positive numbers that satisfies three triangle inequalities. 
\begin{enumerate}
\item We have $\sigma(\triangle(a,b,c))=X+iY$, where \setlength\arraycolsep{1pt}
\begin{eqnarray}
X&=&\displaystyle \frac{a^2+c^2-b^2}{2a^2}, \label{eq_X}\\[4mm]
Y&=&\displaystyle \frac{\sqrt{(a+b+c)(-a+b+c)(a-b+c)(a+b-c)}}{2a^2} \label{eq_Y1}\\[2mm]
&=&\displaystyle \frac{\sqrt{-a^4-b^4-c^4+2a^2b^2+2b^2c^2+2c^2a^2\phantom{{l}^{l}}\!\!\!}}{2a^2}. \label{eq_Y2}
\end{eqnarray}
\item We have 
\begin{equation}\label{phi(triangle(a,b,c))}
\phi(\triangle(a,b,c))=\frac{-2a^2+b^2+c^2+\sqrt3\left(b^2-c^2\right)i}{a^2+b^2+c^2+\sqrt3\sqrt{-a^4-b^4-c^4+2a^2b^2+2b^2c^2+2c^2a^2\phantom{{l}^{l}}\!\!\!}}.
\end{equation}
\end{enumerate}
\end{lemma}

\begin{proof}
(1) The equality \eqref{eq_X} follows from the cosine formula and \eqref{eq_Y1} from Heron's formula. 

\medskip
(2) The equality \eqref{phi(triangle(a,b,c))} can be obtained by substituting  \eqref{eq_X} and \eqref{eq_Y2} into 
\[
\phi=\frac{\sigma-\rho}{\sigma-\bar\rho}
=\frac{\left(X-\frac12\right)^2+Y^2-\frac34-\sqrt3\left(X-\frac12\right)i}{\left(X-\frac12\right)^2+\left(Y+\frac{\sqrt3}2\right)^2}.
\]
\end{proof}

\begin{theorem}\label{main_thm2}
Two bijections from the space of similarity classes of triangles $\widetilde{\mathcal{S}}$ to the unit open disc $\mathcal{D}$, $\phi$ given by \eqref{def_phi} and $\phi_P$ given by \eqref{def_phi_P} satisfy $\phi\circ{\phi_P}^{-1}(z)=i \bar z$. 
\end{theorem}

\begin{proof}
Let $f$ be a map  from $\mathcal{D}$ to $\mathcal{D}$ defined by 
\[
f(z)=\phi\left(\triangle(|z-A_0|, |z-B_0|, |z-C_0|)\right). 
\]
Let $z=x+yi$ $(x,y\in\RR)$. 
Put $a=|z-A_0|, b=|z-B_0|$, and $c=|z-C_0|$, then 
\setlength\arraycolsep{1pt}
\[
\begin{array}{rcl}
a^2&=&\displaystyle {|z-A_0|}^2=x^2+(y-1)^2,  \\[2mm]
b^2&=&\displaystyle {|z-B_0|}^2=\left(x+\frac{\sqrt3}2\right)^2+\left(y+\frac12\right)^2,   \\[4mm]
c^2&=&\displaystyle {|z-C_0|}^2=\left(x-\frac{\sqrt3}2\right)^2+\left(y+\frac12\right)^2 
\end{array}
\]
which implies 
\[
\begin{array}{rcl}
\displaystyle -2a^2+b^2+c^2+\sqrt3\left(b^2-c^2\right)i &=& 6y+6xi, \\[2mm]
a^2+b^2+c^2 &=& 3(x^2+y^2+1), \\[2mm]
-a^4-b^4-c^4+2a^2b^2+2b^2c^2+2c^2a^2 &=& 3\left(1-x^2-y^2\right)^2. 
\end{array}
\]
Substituting the above to \eqref{phi(triangle(a,b,c))} we obtain 
$f(x+yi)=y+xi$, namely $f$ is a bijection given by $f(z)=i \bar z$. 
(We remark that this gives an alternative proof of Theorem \ref{main_thm1} (1) since $\phi$ induces a bijection between ${\mathcal{S}}$ and ${\mathcal{D}}$.) 

Now we can see $f=\phi\circ{\phi_P}^{-1}$, which completes the proof. 
\end{proof}

\section{Two operations that generate similarity classes}
\subsection{Rotation and homothety of the moduli space}
Let $R_\theta$ and $H_\lambda$ be two two operations on $\widetilde{\mathcal{S}}$ which correspond to the rotation by $\theta$ and homothety by $e^\lambda$ in the moduli space $\mathcal{D}_P$ through the bijection $\phi_P$. 
Then they generate $\widetilde{\mathcal{S}}$, to be precise, if we fix a non-equilateral triangle $\triangle A_1B_1C_1$ then for any triangle $\triangle ABC$ there are unique $\theta$ $(0\le\theta<2\pi)$ and $\lambda$ $\,(-\infty\le\lambda<-\log\left|\phi_P([\triangle A_1B_1C_1])\right|\,)$ such that $[\triangle ABC]=H_\lambda R_\theta([\triangle A_1B_1C_1])$. 
The operations $R_\theta$ and $H_\lambda$ 
satisfy 
\[
R_\theta\circ H_\lambda=H_\lambda\circ R_\theta, \hspace{0.4cm} R_{\theta}\circ R_{\theta'}=R_{\theta+\theta'}, \hspace{0.4cm} H_\lambda\circ H_{\lambda'}=H_{\lambda+\lambda'}.
\]

Fixing an edge $BC$, we can describe the operations $R_\theta$ and $H_\lambda$ using pencils of circles introduced in Subsection \ref{subsection_pencils}. 
Let $\Omega_A$ and $\overline\Omega_A$ be given by $\Omega_A=B+\rho(C-B)$ and $\overline\Omega_A=B+\bar\rho(C-B)$ so that $\triangle \Omega_ABC$ and $\triangle \overline\Omega_ACB$ are equilateral triangles.  

Note that a circle with center the origin and a radial ray from the origin in Pompeiu moduli space $\mathcal{D}_P$ correspond to a circle of Apollonius in the upper half plane $\mathcal{H}_+$ with foci $\rho, \bar\rho$ and a subset of a circle through $\rho$ and $\bar\rho$ from $\rho$ to the real axis in $\mathcal{H}_+$ respectively. 
In fact, the bijection $\sigma\circ{\phi_P}^{-1}$ from $\mathcal{D}_P$ to $\mathcal{H}_+$ is given by 
\[
\sigma\circ{\phi_P}^{-1}=(\sigma\circ{\phi}^{-1})\circ(\phi\circ{\phi_P}^{-1})\colon
z\mapsto\frac{\bar\rho(i\bar z)-\rho}{(i\bar z)-1}\,,
\]
which is a M\"obius transformation of $\RR^2\cup\{\infty\}$ that maps $0$ and $\infty$ to $\rho$ and $\bar\rho$ respectively. 
Therefore it maps hyperbolic pencil and elliptic pencil defined by $0$ and $\infty$ to those defined by $\rho$ and $\bar\rho$ (see Subsection \ref{subsection_pencils}). 
Therefore we have

\begin{proposition}
If we fix the edge $BC$, the rotation $R_\theta$ can be realized by moving the vertex $A$ along a circle of Apollonius with foci $\Omega_A$ and $\overline{\Omega}_A$, whereas the homothety $H_\lambda$ can be realized by moving the vertex $A$ along a circle through $\Omega_A$ and $\overline{\Omega}_A$ as illustrated in Figure \ref{fig_operations}. 
\end{proposition}

\begin{figure}[htbp]
\begin{center}
  \psfrag{A}{$A$}
  \psfrag{B}{$B$}
  \psfrag{C}{$C$}
  \psfrag{a}{$A'$}
  \psfrag{s}{$\Omega_A$}
  \psfrag{t}{$\overline{\Omega}_A$}
  \psfrag{d}{$A''$}
  \psfrag{w}{$\theta$}
\includegraphics[width=.36\linewidth]{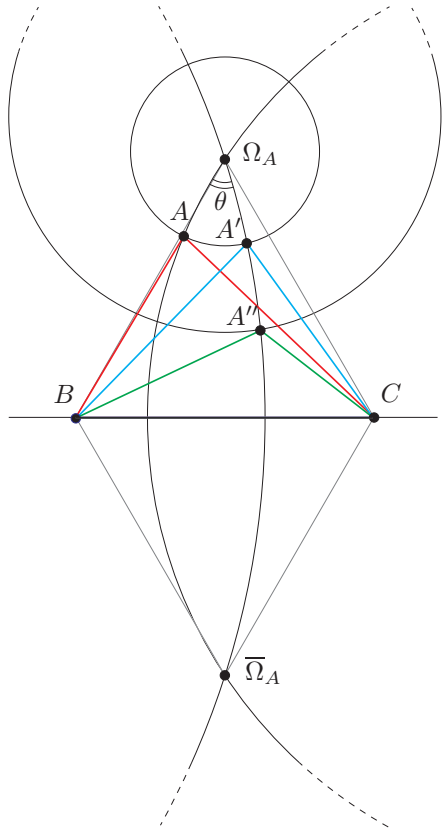}
\caption{Operation realizing $R_\theta$ and $H_\lambda$ such that $H_\lambda\circ R_\theta$ maps $[\triangle ABC]$ to $[\triangle A''BC]$}
\label{fig_operations}
\end{center}
\end{figure}

\subsection{M\"obius geometric interpretation}
The parameter $\theta$ and $\lambda$ have the following M\"obius geometric meaning. 

We work in the Minkowski space $\RR^4_1$ with Lorentz metric, i.e with a pseudo-inner product $\langle \>,\>\rangle$ with signature $-,+,+,+$ given by 
\[
\langle \vect x,\vect y\rangle=-x_0y_0+x_1y_1+x_2y_2+x_3y_3, \hspace{0.4cm}\vect x=(x_0,x_1,x_2,x_3), \vect y=(y_0,y_1,y_2,y_3).
\]
Then $\RR^2\cup\{\infty\}$ can be identified with the set of lines in the light cone through the origin, and a circle (we remark that a line is considered as a circle that passes through $\infty$ in our setting) $\Gamma$ in $\RR^2$ with the intersection of the light cone with a time-like $3$-dimensional vector subspace $\Pi_\Gamma$. 
Let us fix an orientation of $\Gamma$, which fix an orientation of $\Pi_\Gamma$. 
By taking the positive unit normal vector to $\Pi_\Gamma$, any oriented circle (or line) can be identified with a point in the de Sitter space $\Lambda=\{\vect x\,|\,\langle \vect x,\vect x\rangle=1\}$, which is a ``{\sl unit sphere}'' with respect to the Lorentz metric. 
Let $\gamma$ denote a point in $\Lambda$ that corresponds to an oriented circle $\Gamma$. 

The Poncelet pencil of circles with limit points $\Omega$ and $\overline\Omega$ is a set of Apollonius circles with foci  $\Omega$ and $\overline\Omega$. 
Let $\mbox{Span}\langle\Omega,\overline\Omega\rangle$ be a $2$-dimensional vector subspace of $\RR^4_1$ spanned by two lines in the light cone that correspond to $\Omega$ and $\overline\Omega$. 
Then the Poncelet pencil corresponds to the intersection 
\[\Lambda\cap\mbox{Span}\langle\Omega,\overline\Omega\rangle
=\{\vect x\in\mbox{Span}\langle\Omega,\overline\Omega\rangle\,:\, \langle \vect x,\vect x\rangle=1\}.\]
The restriction of the Lorentz metric of $\RR^4_1$ to $\mbox{Span}\langle\Omega,\overline\Omega\rangle$ is again Lorentzian (indefinite), i.e. $\mbox{Span}\langle\Omega,\overline\Omega\rangle$ is a Minkowski plane. 
The intersection $\Lambda\cap\mbox{Span}\langle\Omega,\overline\Omega\rangle$ consists of two hyperbolas, which are time-like geodesics in $\Lambda$. 

Suppose oriented circles $\Gamma_1$ and $\Gamma_2$ in this pencil\footnote{We assume that $\Gamma_1$ and $\Gamma_2$ have the ``same '' orientation so that the corresponding points $\gamma_1$ and $\gamma_2$ in $\Lambda$ belong th the same connected component.} 
are mapped to concentric circles such that the ratio of the radii is equal to $e^\lambda$ by an inversion in a circle with center $\Omega$ (or $\overline\Omega$). 
Then the corresponding points $\gamma_1$ and $\gamma_2$ in $\Lambda$ satisfy $\langle \gamma_1,\gamma_2\rangle=\cosh\lambda$. 
The parameter $\lambda$ serves as the arc-length along this geodesic. 

The pencil of circles with base points $\Omega$ and $\overline\Omega$ is a set of circles that pass both $\Omega$ and $\overline\Omega$. 
Let $\big(\mbox{Span}\langle\Omega,\overline\Omega\rangle\big)^\perp$ be the pseudo-orthogonal complement of $\mbox{Span}\langle\Omega,\overline\Omega\rangle$. 
Then the pencil of circles with base points $\Omega$ and $\overline\Omega$ corresponds to the intersection 
\[\Lambda\cap\big(\mbox{Span}\langle\Omega,\overline\Omega\rangle\big)^\perp
=\big\{\vect x\in\big(\mbox{Span}\langle\Omega,\overline\Omega\rangle\big)^\perp\,:\, \langle \vect x,\vect x\rangle=1\big\}.\]
The restriction of the Lorentz metric of $\RR^4_1$ to $\big(\mbox{Span}\langle\Omega,\overline\Omega\rangle\big)^\perp$ is Riemannian (positive definite). 
The intersection $\Lambda\cap\big(\mbox{Span}\langle\Omega,\overline\Omega\rangle\big)^\perp$ is a circle, which is a space-like geodesic in $\Lambda$. 

Suppose oriented circles $\Gamma_1$ and $\Gamma_2$ intersect in angle $\theta$. Then the corresponding points $\gamma_1$ and $\gamma_2$ in $\Lambda$ satisfy $\langle \gamma_1,\gamma_2\rangle=\cos\theta$. 
The parameter $\theta$ serves as the arc-length along this geodesic. 

\begin{figure}[htbp]
\begin{center}
\begin{minipage}{.47\linewidth}
\begin{center}
\includegraphics[width=0.6\linewidth]{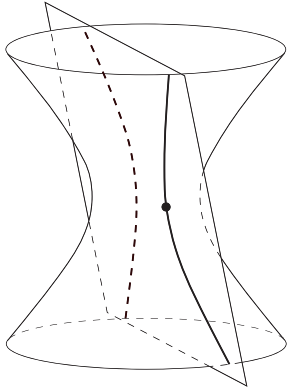}
\caption{Time-like geodesic in de Sitter space $\Lambda$}
\label{geodesic-t}
\end{center}
\end{minipage}
\hskip 0.1cm
\begin{minipage}{.47\linewidth}
\begin{center}
\includegraphics[width=0.6\linewidth]{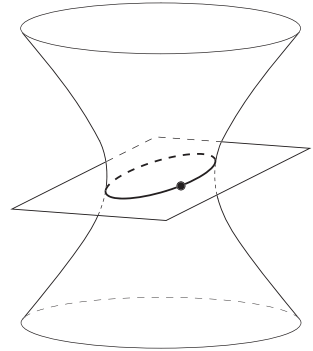}
\caption{Space-like geodesic in de Sitter space $\Lambda$}
\label{geodesic-s}
\end{center}
\end{minipage}
\end{center}
\end{figure}

The reader is referred to \cite{LO} or Chapters 13 and 14 of \cite{L}, for example, for the details. 

\section{Equidivision and Routh operations}\label{section_remaks_operations}
We introduce the results of Nakamura-Oguiso \cite{NO0} to realize the rotation $R_\theta$ of the moduli space by kinds of operations given as follows. 

The first one is a ``{\em $q$-equidivision operation}'' $T_q$ which maps $\triangle ABC$ to $\triangle A''B''C''$, where $A'', B''$, and $C''$ are given by 
\begin{equation}\label{def_equidivision_operator}
A''=qB+(1-q)C, \>\> B''=qC+(1-q)A, \>\> C''=qA+(1-q)B. 
\end{equation}
The second one is a ``{\em $q$-Routh operation}'' $T^R_{q}$ $(q\ne\frac12)$ which maps $\triangle ABC$ to $\triangle A'B'C'$, where $A', B'$, and $C'$ are intersection points of pairs of lines $AA''$ and $BB''$, $BB''$ and $CC''$, and $CC''$ and $AA''$ respectively. 
These two operations are also studied in \cite{H}. 
The third one, denoted by $T_{p,q}$, is a generalization of the above two. 
Put \[
A''_p=(1-p)B+pC, \>\> B''_p=(1-p)C+pA, \>\> C''_p=(1-p)A+pB,
\]
then $T_{p,q}$ maps $\triangle ABC$ to $\triangle A'B'C'$, where $A', B'$, and $C'$ are intersection points of pairs of lines $AA''$ and $BB''_p$, $BB''$ and $CC''_p$, and $CC''$ and $AA''_p$ respectively. 
Remark that the first two operations can be obtained as special cases of $T_{p,q}$. 
In fact, a $q$-equidivision operation $T_q$ corresponds to the case when $p=0$ and a $q$-Routh operation $T^R_{q}$ to $p=1-q$. 

\begin{figure}[htbp]
\begin{center}
\begin{minipage}{.45\linewidth}
\begin{center}
  \psfrag{A}{$A$}
  \psfrag{B}{$B$}
  \psfrag{C}{$C$}
  \psfrag{d}{$A''$}
  \psfrag{e}{$B''$}
  \psfrag{f}{$C''$}
  \psfrag{q}{$q$}
  \psfrag{Q}{$1-q$}
\includegraphics[width=0.8\linewidth]{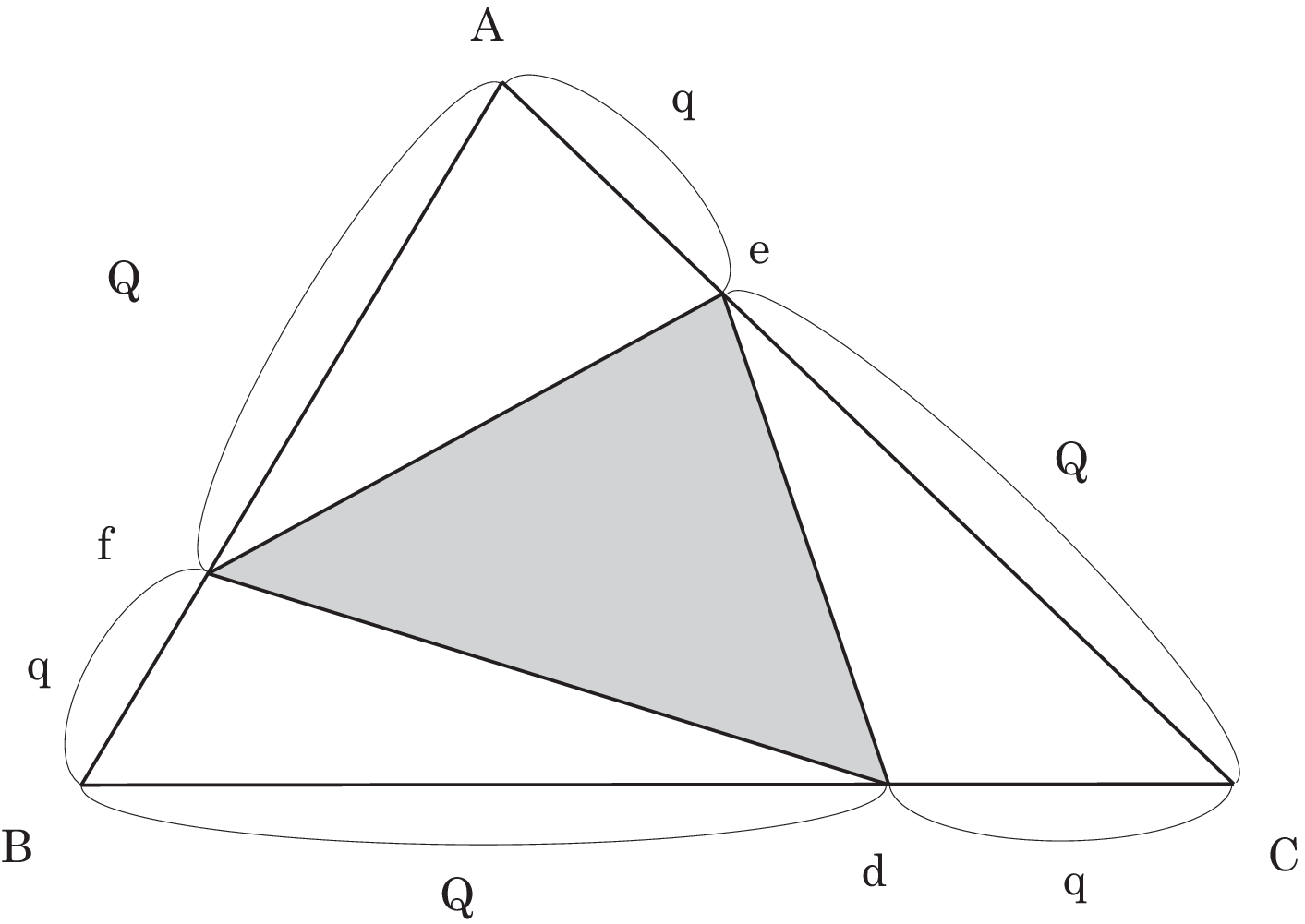}
\caption{$q$-equidivision operation $T_q$}
\label{fig_T_q}
\end{center}
\end{minipage}
\hskip 0.4cm
\begin{minipage}{.45\linewidth}
\begin{center}
  \psfrag{A}{$A$}
  \psfrag{B}{$B$}
  \psfrag{C}{$C$}
  \psfrag{a}{$A'$}
  \psfrag{b}{$B'$}
  \psfrag{c}{$C'$}
  \psfrag{d}{$A''$}
  \psfrag{e}{$B''$}
  \psfrag{f}{$C''$}
  \psfrag{p}{$p$}
  \psfrag{q}{$q$}
  \psfrag{Q}{$1-q$}
\includegraphics[width=0.8\linewidth]{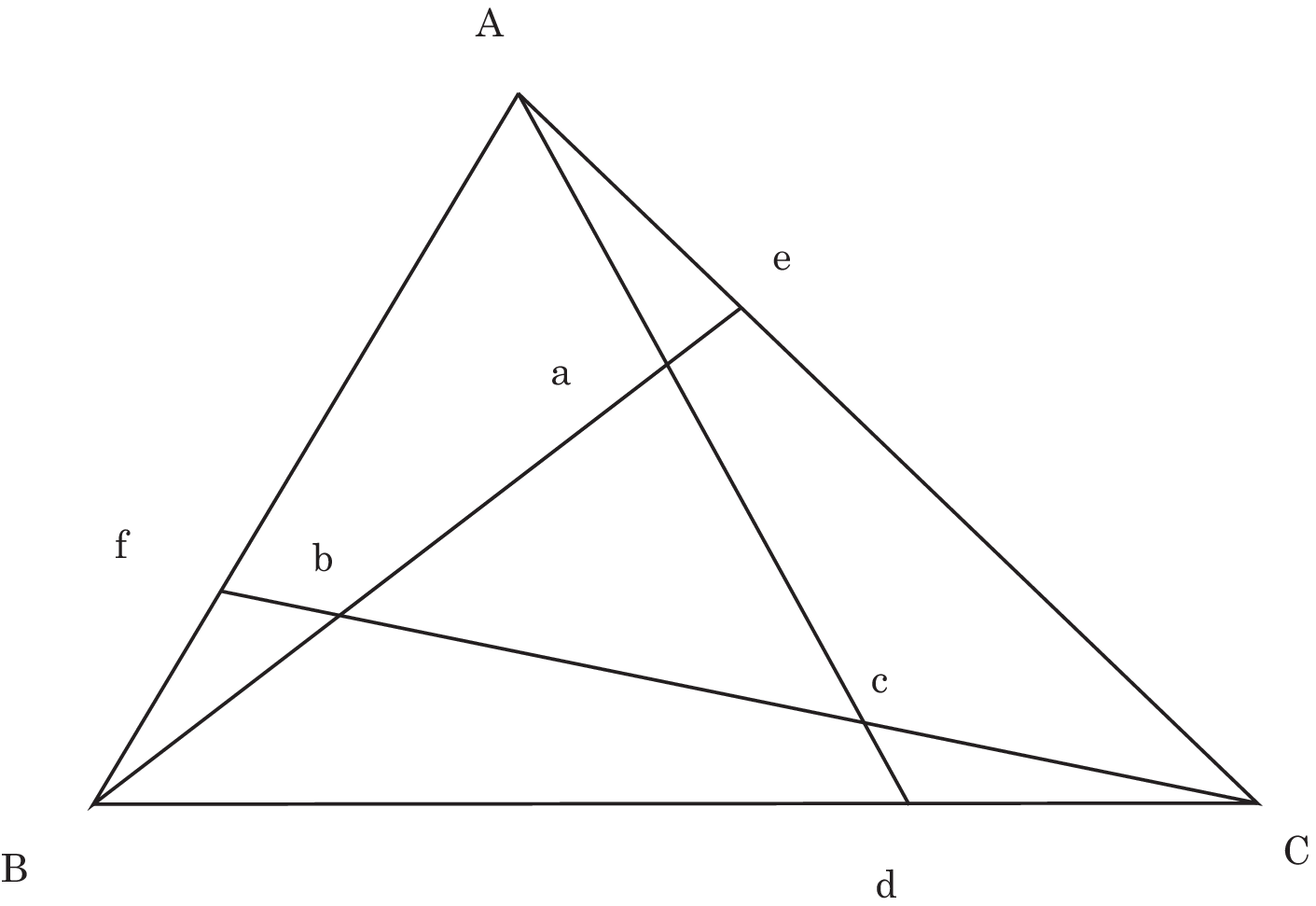}
\caption{Routh operation $T^R_{q'}$ $(q'\ne\frac12)$}
\label{fig_S_q}
\end{center}
\end{minipage}
\end{center}
\end{figure}

\begin{figure}[htbp]
\begin{center}
  \psfrag{A}{$A$}
  \psfrag{B}{$B$}
  \psfrag{C}{$C$}
  \psfrag{a}{$A'$}
  \psfrag{b}{$B'$}
  \psfrag{c}{$C'$}
  \psfrag{d}{$A''$}
  \psfrag{e}{$B''$}
  \psfrag{f}{$C''$}
  \psfrag{D}{$A''_p$}
  \psfrag{E}{$B''_p$}
  \psfrag{F}{$C''_p$}
  \psfrag{p}{$p$}
  \psfrag{q}{$q$}

\includegraphics[width=0.4\linewidth]{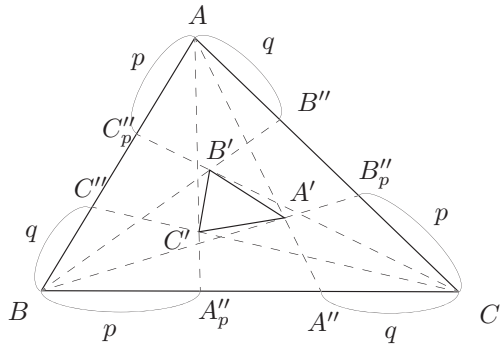}
\caption{The operation $T_{p,q}$}
\label{fig_R_pq}
\end{center}
\end{figure}

The proof of Theorem 1 of \cite{NO0} implies 
\begin{theorem}\label{thm_NO}
{\rm (\cite{NO0}\footnote{We remark that Nakamura Oguiso studied ``{\em shape similarity}\,''of triangles, i.e. similarity without labeling vertices, and hence they used $\phi^3$ to construct the moduli space. This is the reason of the difference of the coefficients, $6$ in \cite{NO0} and $2$ here.})} 
We have
\begin{equation}\label{argument_T_pq}
\phi([T_{p,q}\triangle ABC])=e^{i\theta(p,q)}\phi([\triangle ABC]),
\end{equation}
where 
\begin{equation}\label{theta_pq}
\theta(p,q)=2\arg\{(p-1)(2q-1)\rho-(q-1)(2p-1)\}.
\end{equation}
\end{theorem}

It follows that the equidivision operation $T_q$ and the Routh operation $T^R_{q'}$ correspond to rotations by $2\arctan\left(\sqrt3\,(2q-1)\right)$ and $2\arctan\left(\sqrt3\,q'/(2-q')\right)$ respectively in the moduli space $(\mathcal{D},\phi)$. 
Therefore, given $\theta$, the values $q, q'$ so that $T_q$ and $T^R_{q'}$ realize $R_\theta$ are compass-and-straightedge constructible.

\bigskip

Jun O'Hara

Department of Mathematics and Informatics,Faculty of Science, 
Chiba University

1-33 Yayoi-cho, Inage, Chiba, 263-8522, JAPAN.  

E-mail: ohara@math.s.chiba-u.ac.jp

\end{document}